\documentclass[onefignum,onetabnum]{siamart250211}

\usepackage{lipsum}
\usepackage{amsfonts}
\usepackage{graphicx}
\usepackage{epstopdf}
\usepackage{subcaption}

\usepackage{booktabs}

\usepackage{setspace}
\ifpdf
  \DeclareGraphicsExtensions{.eps,.pdf,.png,.jpg}
\else
  \DeclareGraphicsExtensions{.eps}
\fi

\newsiamremark{remark}{Remark}
\newsiamremark{hypothesis}{Hypothesis}
\crefname{hypothesis}{Hypothesis}{Hypotheses}
\newsiamthm{claim}{Claim}
\newsiamthm{assumption}{Assumption}

\title{An Example Article\thanks{Submitted to the editors DATE.
\funding{This work was funded by the Fog Research Institute under contract no.~FRI-454.}}}

\usepackage{amsopn}

\makeatletter
\newcommand*{\addFileDependency}[1]{
  \typeout{(#1)}
  \@addtofilelist{#1}
  \IfFileExists{#1}{}{\typeout{No file #1.}}
}
\makeatother

\renewcommand{\thesection}{\arabic{section}}
\renewcommand{\thesubsection}{\thesection.\arabic{subsection}}

\usepackage{cite}  

\newcommand{\ckopt}{c^{\text{k-opt}}}
\newcommand{\ckoptT}{c^{\text{k-opt}^{\top}}}
\newcommand{\lambdakopt}{\lambda^{\text{k-opt}}}

\newlength{\stdfigwidth}
\begin{document}

\title{The Convergence and Error Analysis of Coordinate Descent Methods with compression for Full Configuration Interaction}
\author{Yingzhou Li\thanks{School of Mathematical Sciences, Shanghai Key Laboratory for Contemporary Applied Mathematics, Fudan University and
Key Laboratory of Computational Physical Sciences, Ministry of Education (\email{yingzhouli@fudan.edu.cn})}  \and Qiang Wu\thanks{School of Mathematical Sciences, Fudan University (\email{wuq23@m.fudan.edu.cn})}}

\maketitle

\begin{abstract}
We study the effect of compression in Coordinate Descent Full Configuration Interaction (CDFCI) within an unconstrained optimization formulation of the full configuration interaction ground-state problem. Under suitable local assumptions, we prove that the compressed iteration converges linearly to the solution of an associated restricted problem. We also characterize the convergence point of the compressed algorithm. Under an additional exponential decay assumption on the target eigenvector, we show that the resulting eigenvalue error is of order \(O(\tau^2)\), where \(\tau\) denotes the compression threshold. Numerical results support the analysis.
\end{abstract}

\section{Introduction}

Full Configuration Interaction (FCI) provides the exact solution of the electronic Schrödinger equation within a finite orbital basis. Under the Slater determinant basis, the FCI problem can be formulated as the eigenvalue problem
\begin{equation}
Hc=\lambda c,
\end{equation}
where $H$ is the Hamiltonian matrix and the smallest eigenvalue corresponds to the ground-state energy. The dimension of $H$ grows combinatorially with the number of electrons and orbitals, making direct eigensolvers impractical except for very small systems. A number of approaches have therefore been developed for large-scale FCI calculations, including Davidson-type methods, selected configuration interaction methods, and stochastic projector methods~\cite{DAVIDSON197587,huron1973iterative,holmes2016heat,booth2009fermion,booth2013towards,lu2020full}

Among these methods, Coordinate Descent Full Configuration Interaction (CDFCI), proposed in~\cite{wang2019coordinate}, reformulates the eigenvalue problem as an unconstrained optimization problem and updates one coordinate at a time. Since each iteration only accesses a small portion of the Hamiltonian, CDFCI avoids the storage and computational costs associated with full-vector updates and is therefore suitable for very large FCI problems~\cite{wang2019coordinate,wang2023coordinate,zhang2025parallel}.

A compression step is commonly used in practical implementations of CDFCI. During the iteration, the support of the coefficient vector typically grows continuously. Without compression, both storage and computational cost increase rapidly as the iteration proceeds. To control the size of the iterate, entries whose magnitudes are below a prescribed threshold are discarded. The resulting vector remains sparse, allowing the algorithm to be applied to larger systems.

The compression step, however, makes the analysis substantially more difficult. Most convergence analyses of coordinate descent methods assume that the iteration is performed in the full coordinate space and do not account for compression or truncation of the iterates~\cite{nesterov2012efficiency,richtarik2014iteration,wright2015coordinate}. In compressed CDFCI, however, coordinates may be repeatedly removed and reintroduced during the iteration, and the resulting algorithm falls outside the standard coordinate descent framework. Consequently, the existing theory cannot be directly applied to compressed CDFCI.

Compression and truncation have also been used in eigenvalue computations. Representative examples include truncated power methods, sparse PCA algorithms, compressed subspace iterations, and Krylov-based eigensolvers~\cite{yuan2013truncated,liuAnalysisTruncatedOrthogonal2021,d2004direct,greene2022approximating,saad2025global,casulli2026lanczos}. Most existing analyses rely on the assumption that the target eigenvector is sparse or approximately sparse. In contrast, the exact FCI eigenvector is generally not sparse, and compression is introduced solely to control the size of the iterate. Therefore, these results do not directly apply to CDFCI with compression.

The effect of compression is also related to eigenvalue perturbation theory~\cite{chenPERTURBATIONANALYSISEIGENVECTOR2012,fanEigenvectorPerturbationBound,jiaConvergenceAnalysisInexact2008}. Classical perturbation results describe the change of eigenvalues and eigenvectors under matrix perturbations. In the present setting, however, the compression error is generated dynamically during the iteration and cannot be represented by a fixed perturbation matrix. This requires a different analysis.

The purpose of this paper is to study the effect of compression in CDFCI. We establish convergence results for the compressed iteration and quantify the error introduced by compression.

The main contributions of this work are summarized as follows.

\begin{enumerate}
\item We prove that the compressed CDFCI iteration converges linearly to the solution of an associated restricted optimization problem.

\item We characterize the limiting point of the compressed iteration and show its relation to the corresponding restricted eigenvalue problem.

\item We derive an eigenvalue error estimate caused by compression. In particular, we prove that the error between the ground-state eigenvalue \(\lambda\) of the full problem and the optimal eigenvalue \(\lambda_{k\text{-opt}}\) of the compressed \(k\)-dimensional problem satisfies
\[
|\lambda-\lambda_{k\text{-opt}}|=O(\tau^2),
\]
where \(\tau\) denotes the compression threshold.
\end{enumerate}

The remainder of this paper is organized as follows. Section~2 reviews the coordinate descent algorithm with compression. Section~3 presents the main theoretical results, including the linear convergence rate and an upper bound on the energy error, and outlines the key ideas of the proofs. Section~4 is devoted to the rigorous proofs of these results. Section~5 reports numerical experiments that validate the theoretical findings and illustrate the practical performance of the algorithm. Finally, Section~6 concludes the paper and discusses possible directions for future research.

\section{Preliminaries on CDM and Compression }
\subsection{Dimensionality and Sparsity Structure}

The dimension of the FCI Hamiltonian matrix is determined by the number of spin orbitals \(M\) and electrons \(N\), scaling according to the binomial coefficient \(\binom{M}{N}\). Consequently, the matrix dimension grows exponentially with respect to the system size. This combinatorial growth is illustrated in \cref{tab:fci-summary}, which reports the dimensions of the FCI Hamiltonian for the water molecule (H\(_2\)O, \(N=10\)) under various basis sets.

Despite this exponential growth in dimension, the FCI Hamiltonian matrix remains extremely sparse due to the structure of the second-quantized Hamiltonian: a matrix element is nonzero only when the corresponding pair of Slater determinants differs by at most two spin orbitals. As a consequence, the number of nonzero entries per row scales only polynomially with \(M\) and \(N\). The last column of \cref{tab:fci-summary} reports an estimated upper bound on the maximum number of nonzero entries per row, which is derived from the counts of single and double excitations:
\[
N(M-N) + \frac{N(N-1)}{2} \cdot \frac{(M-N)(M-N-1)}{2}.
\]
\begin{table}[h]
\centering
\begin{tabular}{lcccc}
\hline
Basis set & Spin orbitals $M$ & FCI dimension & Max nnz per row \\
\hline
STO-3G   & 24  & $2.0\times 10^6$      & 4,235 \\
cc-pVDZ  & 48  & $6.54\times 10^9$     & 32,015 \\
cc-pVTZ  & 82  & $2.13\times 10^{12}$  & 115,740 \\
cc-pVQZ  & 140 & $5.73\times 10^{14}$  & 378,625 \\
\hline
\end{tabular}
\caption{FCI Hamiltonian dimension and estimated maximum number of nonzero elements per row for H$_2$O ($N=10$) under different basis sets.}
\label{tab:fci-summary}
\end{table}
Furthermore, the ground-state eigenvector of the FCI Hamiltonian exhibits exponential decay. Specifically, the magnitudes of its entries, when sorted in descending order, decrease exponentially. This property provides a justification for the truncation strategy, as the tail components make a negligible contribution to the wave function. This decay also plays a central role in establishing sharper error bounds for the compression scheme.

\begin{figure}[h]
    \centering
    \includegraphics[width=\stdfigwidth]{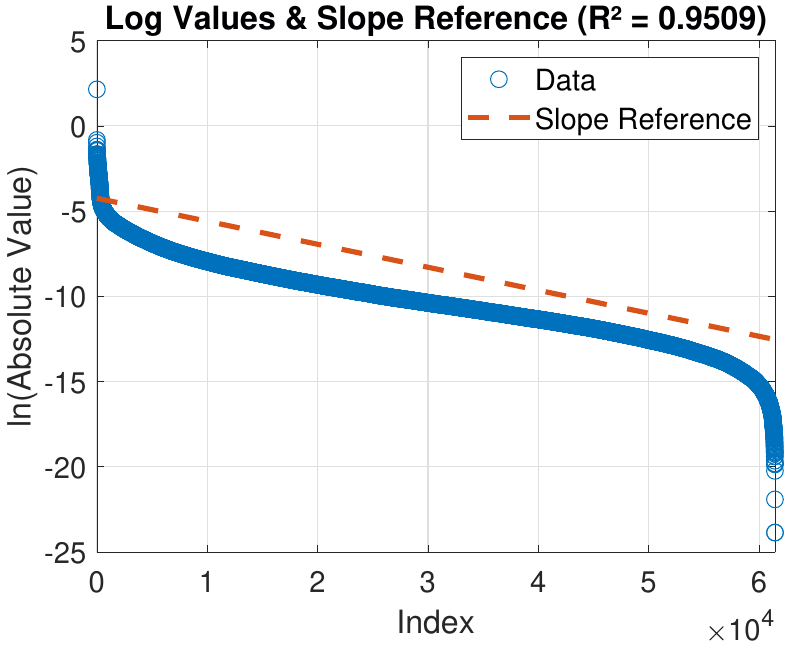}
    \caption{Exponential decay of the ground-state eigenvector components for H$_2$O-STO3G basis.}
    \label{fig:exp-decay}
\end{figure}

\subsection{Coordinate Descent with Compression}
\label{subsec:cdm-compression}

The FCI eigenvalue problem can be recast as the following unconstrained nonconvex optimization problem:
\begin{equation}
\label{eq:optimization_problem}
    \min_{c \in \mathbb{R}^{N_{\mathrm{FCI}}}} f(c) = \left\| H + c c^{\top} \right\|_{\mathrm{F}}^2,
\end{equation}
where \(\|\cdot\|_{\mathrm{F}}\) denotes the Frobenius norm and \(N_{\mathrm{FCI}}\) is the dimension of the FCI Hamiltonian matrix. The gradient of the objective function is given by
\begin{equation}
    \nabla f(c) = 4 H c + 4 \left( c^{\top} c \right) c.
\end{equation}
As established in~\cite{li2019coordinatewise}, let
\[
\lambda=\lambda_0<\lambda_1\le\lambda_2\le\cdots
\]
be the eigenvalues of \(H\), and let \(c_i\) denote the normalized eigenvector corresponding to \(\lambda_i\). The stationary points of~\eqref{eq:optimization_problem} are precisely \(0\) and
\[
\pm\sqrt{-\lambda_i}\,c_i,\qquad i=0,1,\ldots.
\]
Moreover, only the points \(\pm\sqrt{-\lambda}\,c_0\) are local minimizers, and both attain the same global minimum value. All other stationary points are saddle points or local maximizer. Consequently, solving~\eqref{eq:optimization_problem} yields both the ground-state energy \(\lambda\) and the corresponding ground-state eigenvector \(c_0\).

To solve the optimization problem~\eqref{eq:optimization_problem}, traditional eigenvalue solvers and general-purpose optimization algorithms are computationally intractable due to the exponential dimensionality of the Hamiltonian matrix. Consequently,~\cite{li2019coordinatewise} proposed a coordinate descent framework for solving this optimization problem.
    
The Coordinate Descent FCI (CDFCI) algorithm maintains two sparse vectors, \(c^{(l)}\) and \(b^{(l)}\), throughout the iterations. Here, \(c^{(l)}\) denotes the coefficient vector computed at the \(l\)-th iteration, while \(b^{(l)}\) serves as a compressed approximation of the matrix-vector product \(H c^{(l)}\).

In the \((l+1)\)-th iteration, CDFCI first identifies the coordinate index \(i^{(l+1)}\) corresponding to the gradient component with the largest magnitude:
\begin{equation}
    i^{(l+1)} = \operatorname*{arg\,max}_{j} \left| 4 b_j^{(l)} + 4 \left( (c^{(l)})^{\top} c^{(l)} \right) c_j^{(l)} \right|.
\end{equation}
Subsequently, the algorithm performs an exact line search to determine the optimal update. To facilitate the compression of \(b\), we formulate this step as finding the optimal value for the selected coordinate, rather than an additive increment. Specifically, the new value \(\alpha\) is determined by
\begin{equation}
    \alpha = \operatorname*{arg\,min}_{\widetilde{\alpha} \in \mathbb{R}} \; f\left( c^{(l)} - c_{i^{(l+1)}}^{(l)} e_{i^{(l+1)}} + \widetilde{\alpha} e_{i^{(l+1)}} \right),
\end{equation}
where \(e_{i^{(l+1)}}\) denotes the standard basis vector corresponding to index \(i^{(l+1)}\). This update rule implies that the entry \(c_{i^{(l+1)}}\) is implicitly reset to zero and then assigned the new value \(\alpha\) that minimizes the objective function.

The vector \(b^{(l)}\) plays a pivotal role in ensuring computational efficiency. Specifically, once the new coefficient vector \(c^{(l+1)}\) is determined, \(b^{(l)}\) is updated recursively to obtain \(b^{(l+1)}\). The update rule is given by
\begin{equation}
    b^{(l+1)} = b^{(l)} - c_{i^{(l+1)}}^{(l)} H_{:, i^{(l+1)}} + \alpha H_{:, i^{(l+1)}}.
\end{equation}
This incremental update strategy avoids the expensive recomputation of the full matrix-vector product \(H c\). Instead, it efficiently adjusts \(b^{(l)}\) by subtracting the contribution of the previous scalar component at coordinate \(i^{(l+1)}\) and adding the contribution from the newly computed value \(\alpha\).

However, updating the vector \(b\) exactly is computationally expensive. To reduce the memory cost,~\cite{wang2019coordinate} introduced a compression strategy that controls when new nonzero entries are added to \(b\). Given a threshold \(\tau\), the following rules are applied to each coordinate \(j\):

\begin{romannum}
    \item \textbf{If \(b_j^{(l)} \neq 0\)} (the coordinate is already selected): The value is always updated, regardless of the magnitude of the change.
    
    \item \textbf{If \(b_j^{(l)} = 0\)} (the coordinate is not selected): The coordinate is updated only if the new term is sufficiently large, that is, if
    \[
    \left| \alpha H_{j,i^{(l+1)}} \right| > \tau.
    \]
    Otherwise, the update is discarded and \(b_j^{(l+1)}\) remains zero.
\end{romannum}

This strategy preserves all existing nonzero entries, while new coordinates are added to the active set only if their contribution is significant.

A coordinate \(i\) remains inactive if it is not selected by the greedy rule. For such a coordinate, \(c_i\) stays zero. The compression strategy ensures that \(b_i\) also remains zero unless it receives a large update. This controls the number of nonzero elements (the support size). Crucially, once a coordinate \(i\) is selected and updated, it enters the active set and is never compressed again. Thereafter, its value is recomputed at every iteration and is no longer affected by the compression threshold.

This implies that we only discard entries in \(b\) for coordinates where \(c_i = 0\). Consequently, the Rayleigh quotient
\begin{equation}
    \frac{(c^{(l)})^{\top} H c^{(l)}}{(c^{(l)})^{\top} c^{(l)}}
\end{equation}
can be computed exactly using the auxiliary vector \(b\):
\begin{equation}
    \frac{(c^{(l)})^{\top} b^{(l)}}{(c^{(l)})^{\top} c^{(l)}}.
\end{equation}
This equality holds because any compressed term in \(b\) corresponds to a zero coefficient in \(c\), so it makes no contribution to the inner product (that is, \(c_i b_i = 0\)).

\section{Main Theoretical Results}

In this section, we present the theoretical analysis of the CDFCI algorithm with compression. Our analysis consists of two main parts. First, we prove the linear convergence of the algorithm by studying the energy decrease in each iteration. Second, we analyze the convergence point of the algorithm and show that the final result approximates the true ground state with a bounded energy error. Together, these results provide theoretical guarantees for both the convergence rate and the solution accuracy.

\subsection{Linear Convergence}

In this subsection, we analyze the convergence of the CDFCI algorithm with compression. We begin by stating the necessary assumptions. Throughout the analysis, we assume that the Hamiltonian matrix \(H\in\mathbb{R}^{n\times n}\) is sparse, and let \(n_z\) denote the maximum number of nonzero entries in any row of \(H\). This assumption is consistent with the structure of the FCI Hamiltonian, as discussed in Section~2.

We start by introducing the necessary definitions and lemmas. These results establish that the objective function is strongly convex in a neighborhood of the optimal solution, which is the key ingredient in the proof of the linear convergence result stated in~\cref{theorem:linear-convergence}.
\begin{definition}
\label{def:strong_convexity}
A continuously differentiable function \( g: \mathbb{S} \to \mathbb{R} \) is said to be strongly \(\|\cdot\|_p\)-convex if there exists a constant \(\mu_p > 0\) such that
\begin{equation}
    g(y) \geq g(x) + \nabla g(x)^{\top}(y-x) + \frac{\mu_p}{2}\|y-x\|_p^2, \quad \forall x, y \in \mathbb{S}.
\end{equation}
\end{definition}

To analyze the local behavior of the objective function near the optimal solutions \(\pm c\), we define the following neighborhoods, following the approach in~\cite{li2019coordinatewise}:
\begin{align*}
    B^{\pm} &= \left\{ y \;\middle|\; \| y \mp c \|_2 \leq \frac{1}{30} \frac{\min(-2\lambda, -\lambda+\lambda_2)}{\sqrt{-\lambda}} \right\}, \\
    D^{\pm} &= \left\{ x \in B^{\pm} \;\middle|\; f(x) \leq \min_{y \in \partial B^{\pm}} f(y) \right\}.
\end{align*}
Here, \(\lambda\) and \(\lambda_2\) denote the smallest and second smallest eigenvalues of the Hamiltonian \(H\), respectively. The following \cref{lem:strong_convexity_constants} shows that the objective function \(f\) is strongly convex within the neighborhoods \(B^{\pm}\).

\begin{lemma}
\label{lem:strong_convexity_constants}
    The function \( f \) is strongly \(\|\cdot\|_2\)-convex on both \( B^{+} \) and \( B^{-} \) with the constant \(\mu_2 = 3 \min(2 |\lambda|, |\lambda-\lambda_2|)\). Furthermore, for any \( p \geq 1 \), there exists a constant \(\mu_p > 0\) such that \( f \) is strongly \(\|\cdot\|_p\)-convex.
\end{lemma}

We assume that the initial vector \( c^{(0)} \) lies in \( D^{\pm} \). This assumption is reasonable in quantum chemistry. Specifically, the Hartree-Fock method typically provides an initial guess that is close to the true ground state. Therefore, the initial coefficient vector is usually close to a minimizer of \( f \), and thus lies inside the strongly convex neighborhood required for our analysis.

Let \( V_0 \) denote the set of nonzero coordinates (the support) of the initial vector \( c^{(0)} \). Let \( V_l \) be the cumulative support set, which contains all nonzero coordinates appearing in the iterates up to the \( l \)-th step. We prove that the total number of active coordinates remains bounded.

\begin{definition}
\label{def:maximal-support}
We define the cumulative support \( V_l \) at iteration \( l \) as
\[
V_l := \bigcup_{k=0}^l \operatorname{supp}(c^{(k)}), \quad l \ge 0,
\]
and the maximal support set \( V \) as
\[
V := \bigcup_{l \ge 0} V_l.
\]
\end{definition}

\begin{lemma}
\label{lemma:support_set}
The maximal support set \( V \) defined in \cref{def:maximal-support} is finite.
\end{lemma}
\begin{proof}
By construction, the sets satisfy the recurrence \(V_l = V_{l-1} \cup \operatorname{supp}(c^{(l)})\). Consequently, \(\{V_l\}\) forms a nested sequence:
\[
V_0 \subseteq V_1 \subseteq V_2 \subseteq \cdots.
\]
Since every \(V_l\) is a subset of the finite index set \(\{1, 2, \ldots, n\}\), their union \(V = \bigcup_{l \ge 0} V_l\) must also be a subset of \(\{1, 2, \ldots, n\}\). Therefore, \(V\) is finite. Moreover, by definition, \(V\) contains the support of \(c^{(l)}\) for all \(l\), and thus serves as the maximal support set throughout the iterations.
\end{proof}

    To establish the linear convergence of the CDFCI algorithm with compression, we proceed in three steps, each is supported by a corresponding lemma. The overall strategy is as follows:
    \begin{enumerate}
        \item Function value gap control: We first show that the objective gap $f({c}^{(l)}) - f({c} )$ is  $O(\tau^2)$ when the step $l$ is sufficiently large as in \cref{lemma:gap}
        \item  Based on the error bound, we prove that the limit of the sequence lies inside the region where the function is strongly convex. This is shown in \cref{lemma:within-area}.
        \item Convergence rate: Finally, by leveraging the strong convexity of the objective in this region, we establish the linear convergence of the algorithm as in \cref{theorem:linear-convergence}.
    \end{enumerate}

\begin{lemma}
\label{lemma:gap}
     Let $c^{(0)} \in D^{\pm}$ be the initial point. When $l$ is sufficiently large, we can obtain
   \[ f(c^{(l)})-f(c)<\frac{8n_z^2\tau^2}{\mu_1}.\]
\end{lemma}
\begin{lemma}
\label{lemma:within-area}
Define the maximal support set $V$ of the initial point $c^{(0)}$
as in \cref{def:maximal-support}.  
Consider the principal submatrix of $H$ restricted to the index set $V \times V$.  
Let $\ckopt$ be the eigenvector corresponding to its smallest eigenvalue $\lambdakopt$, normalized such that $\|\ckopt\|_2 = \sqrt{-\lambdakopt}$.  
When $\tau$ is sufficiently small, $\ckopt$ lies within $D^{\pm}$.

\end{lemma}

\begin{theorem}
\label{theorem:linear-convergence}
    From \cref{lemma:within-area}, it follows that the algorithm with compression converges to  $\ckopt$ or $-\ckopt$  and still exhibits linear convergence, that is:
\[f(c^{(l+1)})-f(\ckopt)\leq(1-\frac{\mu_1}{L})(f(c^{(l)})-f(\ckopt)).\]
\end{theorem}

The proofs of the above results are deferred to Section~4.

\subsection{Analysis of the convergence points}

We now analyze the convergence points. Without loss of generality, we assume that the first \(k\) coordinates are selected during the iteration process. Let \(\ckopt_1 \in \mathbb{R}^k\) denote the sub-vector of \(\ckopt\) corresponding to these selected coordinates.

In addition, since the Hamiltonian matrix \(H\) is extremely sparse and each row contains at most \(n_z\) nonzero entries, we consider the regime \(k > n_z\), which is the regime considered in the subsequent analysis. Accordingly, we partition the matrix \( H \in \mathbb{R}^{n \times n} \) as
\[
H = 
\begin{bmatrix}
H_{11} & H_{12} \\
H_{21} & H_{22}
\end{bmatrix},
\]
where \( H_{11} \in \mathbb{R}^{k \times k} \) corresponds to the block of selected coordinates. Let \( \lambdakopt \) denote the smallest eigenvalue of \( H_{11} \).

It follows from the proof of \cref{theorem:linear-convergence} that the sub-vector \( \ckopt_1 \) is an eigenvector of \( H_{11} \) associated with \( \lambdakopt \), and satisfies the condition \( \|\ckopt_1\|_2 = \sqrt{-\lambdakopt} \).

Since the support of \( \ckopt \) is fixed to the first \( k \) indices, it follows that for all \( i > k \), the coordinate \( i \) is never selected. This implies that during the CDFCI process, we have the following bound:
\begin{equation}
    | (H \ckopt)_i | = \left| \sum_{j=1}^k H_{ij} (\ckopt_1)_j \right| < n_z \tau.
\end{equation}
This inequality holds because whenever the auxiliary vector \( b \) is updated, the value \( b_i \) is compressed to zero due to the thresholding rule. Therefore, the coordinate \( i \) is excluded from selection.

To analyze the convergence, we introduce structural assumptions on the Hamiltonian and its ground-state eigenvector. These assumptions allow us to quantify the effect of the compression strategy.

\begin{assumption}\label{assumption:Spectral-gap}
We assume that the matrix \( H \in \mathbb{R}^{n \times n} \) has a positive spectral gap. Specifically, let \( \lambda \) be the smallest eigenvalue of \( H \). We define the gap as
\[
\operatorname{gap}(H) := \min_{\lambda_H \in \sigma(H), \, \lambda_H \neq \lambda} |\lambda - \lambda_H| > 0.
\]
\end{assumption}

\begin{assumption}\label{assumption:Exponential-decay}
Let \( c \in \mathbb{R}^n \) be the eigenvector corresponding to the smallest eigenvalue \( \lambda \). We assume that \( c \) exhibits exponential decay. Specifically, if we rearrange the indices such that the entries of \( c \) are sorted by magnitude (i.e., \( |c_1| \ge |c_2| \ge \dots \ge |c_n| \)), then there exist constants \( N > 0 \) and \( P > 0 \) such that
\[
|c_j| \leq |c_N| e^{-P(j-N)}, \quad \forall j \geq N.
\]
\end{assumption}

As a direct consequence of \cref{assumption:Exponential-decay}, we establish a technical lemma that provides norm estimates for exponentially decaying vectors. This result will be repeatedly used to control tail contributions in the analysis.

\begin{lemma}
\label{lemma:exp-decay}
Let $ x \in \mathbb{R}^n $ be a vector with exponential decay, meaning that there exist constants $N > 0$ and $P > 0$ such that
\[
|(x)_j| < |(x)_N| e^{-P(j - N)}, \quad \forall j \geq N.
\]
Then there exists a constant $Q > 0$, depending only on the $N$ and $P$ and independent of the dimension  $n$, such that
\[
\|x\|_s < Q \|x\|_\infty, \quad \text{for } s = 1 \text{ or } 2.
\]
\end{lemma}

From \cref{lemma:exp-decay} we can get $\Vert c\Vert_1^2<Q^2\Vert c\Vert_\infty^2\leq Q^2\Vert c\Vert_2^2=-\lambda Q^2.$

Next, we derive a preliminary bound on the eigenvalue error \(\Delta\lambda\).

\begin{lemma}
\label{lemma:O(tau)}
Let \(\Delta \lambda = |\lambda - \lambdakopt|\). Recall from the previous section that for all unselected coordinates \(i\) (where \(k+1 \leq i \leq n\)), the matrix-vector product satisfies \(\vert H_{ij}(\ckopt_1)_j\vert < \tau\). Under this condition, we have
\begin{equation}
    \Delta\lambda < \frac{n_z \|c\|_1 \tau}{-2\lambda} < \frac{n_z Q \tau}{2\sqrt{-\lambda}}.
\end{equation}
\end{lemma}

We now analyze the structure of the true eigenvector \(c\). Consistent with the partition of \(H\), we split \(c\) into two parts: \(c_1 \in \mathbb{R}^k\), which contains the first \(k\) components, and \(c_2 \in \mathbb{R}^{n-k}\), which contains the remaining components. This leads to the following lemma regarding the magnitude of the tail component \( c_2 \).

\begin{lemma}
\label{lemma:norm-of-c2}
Under~\cref{assumption:Spectral-gap} and ~\cref{assumption:Exponential-decay}, we have
\begin{equation}
    \|c_2\|_\infty < \frac{2 m n_z \sqrt{k} \tau}{\operatorname{gap}(H)},
\end{equation}
where the constant \( m \) is defined as
\[
    m = \sqrt{\frac{-\lambdakopt \|c\|_1^2}{4k\lambda^2} + 1}.
\]
Note that \( m \) is bounded by \( \sqrt{\frac{-\lambdakopt Q^2}{4k\lambda} + 1} \) and is independent of the dimension \( n \).
\end{lemma}

Building on the analysis of \( c_2 \), we now establish a refined estimate for the eigenvalue error. We show that the energy error is bounded by a quantity proportional to \( \tau^2 \).

\begin{theorem}
\label{theorem: O(tau^2)}
    Under Assumptions \ref{assumption:Spectral-gap} and \ref{assumption:Exponential-decay}, the eigenvalue error satisfies the following bound:
    \begin{equation}
        \Delta \lambda = |\lambda - \lambdakopt| < \frac{4 m n_z^2 \sqrt{k} Q \tau^2}{-\lambda \operatorname{gap}(H)}.
    \end{equation}
\end{theorem}
\section{Technical proof}
Throughout this section, we assume that the support of \(c^{(l)}\)
(and hence the selected coordinates) is contained in the first \(k\) indices.
Since all arguments are invariant under permutations of coordinates,
we fix this ordering convention throughout the remainder of this section.
Furthermore, for notational simplicity, we assume that \(k>n_z\), where \(n_z\) denotes the maximum number of nonzero entries in any row of \(H\). The general case can be handled by replacing \(n_z\) with \(\min(k,n_z)\) throughout the analysis.

\subsection{\cref{lemma:gap}'s proof}
\begin{proof}

Firstly, at iteration \(l\), after selecting the coordinate \(j_l\), the compression algorithm behaves identically to the original (uncompressed) algorithm. Therefore, Lemma 5.5 in~\cite{li2019coordinatewise} still holds:
\begin{equation}
    f(c^{(l+1)}) \leq f(c^{(l)}) - \frac{1}{2L}\bigl(\nabla_{j_l} f(c^{(l)})\bigr)^2.
\end{equation}
Furthermore, the fact that $c^{(l+1)} \in D^{\pm}$ also follows from the theoretical analysis of the original algorithm. The function $f\left(c^{(l)}\right)$ is monotonically decreasing and has a lower bound $f(c)$, converging to a value $\eta \geq f(c)$. Fixing $0 <\varepsilon<\frac{n_z^2\tau^2}{2L}$, when $l$ is sufficiently large,
 $f\left(c^{(l)}\right)-f\left(c^{(l+1)}\right)<\varepsilon.$ That is $\frac{1}{2 L}\left(\nabla_{j_l} f(c^{(l)})\right)^2<\varepsilon,$ $\left|\nabla_{j_l} f\left(c^{(l)}\right)\right|<\sqrt{2 L \varepsilon}.$
Because $\left|\nabla_{j_l} f\left(c^{(l)}\right)\right|$ attains the maximum absolute value among the first $k$ components  (i.e., the selected coordinates)  of $\nabla f\left(c^{(l)}\right)$,
we have:
\begin{equation}
    \vert\nabla_i f(c^{(l)})\vert<\sqrt{2 L \varepsilon} ,\quad \forall 1 \leq i \leq k.
\end{equation}

When $i>k$, $b^{(l)}_i$ and $c_i=0$, and $\vert\nabla_i f(c^{(l)})\vert=4(Hc^{(l)})_i+(c^{(l)^{\top}}c^{(l)})(c^{(l)})_i=4(Hc^{(l)})_i$. $b^{(l)}_i=0$ indicates that $b_i$ was compressed during the update. Each time we update $c_{j_l}^{(l)}$ using exact line search, it is equivalent to updating from $0$, and we can consider it as having performed 
$k$ compressions. $b^{(l)}_i$ has been compressed at most $k$ times and the $j$-th compression introduces an error less than $\vert 4H_{ij}c^{(l)}_j\vert$.  Thus,
\begin{equation}
\begin{aligned}
\vert\nabla_i f(c^{(l)})\vert
&=4\vert (Hc^{(l)})_i\vert\\
&\leq\vert4b^{(l)}_i\vert+\sum_{j=1}^{k}\vert 4H_{ij}c^{(l)}_j\vert\\
&=\sum_{j=1}^{k}\vert 4H_{ij}c^{(l)}_j\vert
<4\min\{n_z,k\}\tau
=4n_z\tau,\qquad \forall\, i>k,
\end{aligned}
\end{equation}
where the last equality follows from the convention \(k>n_z\) adopted at the beginning of this section.
Since $\varepsilon<\frac{n_z^2\tau^2}{2L}$ and $\sqrt{2L\varepsilon}<n_z\tau$, we have
\begin{equation}
    \| \nabla f(c^{(l)}) \|_{\infty}<4n_z \tau.
\end{equation}

Moreover, as $c^{(l)}$ is updated, this inequality continues to hold.
Due to the strong convexity of $f$ on $B^{\pm}$:
\begin{equation}
\begin{aligned}
    f(c)&\geq f(c^{(l)})-(-\nabla f(c^{(l)})^{\top}(c-c^{(l)})-\frac{\mu_1}{2}\Vert c-c^{(l)}\Vert_1^2)\\
    &\geq f(c^{(l)})-(\Vert \nabla f(c^{(l)})\Vert_{\infty}\Vert c-c^{(l)}\Vert_1-\frac{\mu_1}{2}\Vert c-c^{(l)}\Vert_1^2))\\
    &\geq f(c^{(l)})-\frac{1}{2\mu_1}\Vert \nabla f(c^{(l)})\Vert_{\infty}^2\\
    &=f(c^{(l)})-\frac{8n_z^2\tau^2}{\mu_1},
\end{aligned}   
\end{equation}
where the third inequality is obtained by regarding the expression within the parentheses as a quadratic function of $\left\|c-c^{(l)}\right\|_1$, and finding the maximum value of this quadratic function. 
\end{proof}

Having established that the when $l$ is large the point of the iteration yields a function value close to the global minimum, 
we next check the limit point location. In particular, we show that the limiting point must lie within the region $D^{\pm}$.

\subsection{\cref{lemma:within-area}'s proof}
\begin{proof}
Let \(C_V=\{x\in\mathbb{R}^n : \operatorname{supp}(x)\subseteq V\}\).  
By construction, both \(\ckopt\) and \(c^{(l)}\) lie in \(C_V\).  
Recall that our objective function is
\[
f(x) = \| H + x x^{\top} \|_\mathrm{F}^2 .
\]
For any \(x \in C_V\), the matrix \(xx^{\top}\) has support contained in \(V \times V\); hence the value of \(f(x)\) depends only on the principal submatrix \(H_{VV}\) of \(H\).  
Therefore, minimizing \(f(x)\) over \(C_V\) is equivalent to solving the reduced problem
\[
\min_{z\in \mathbb{R}^{|V|}} \; \| H_{VV} + z z^{\top} \|_\mathrm{F}^2.
\]

In~\cite{li2019coordinatewise} it is shown that, for an objective of the form
\(
\|H + x x^{\top}\|_\mathrm{F}^2,
\)
the local minimum is achieved precisely at the eigenvector corresponding to the smallest eigenvalue.

Applying this result to the reduced problem with \(A = H_{VV}\), we conclude that the minimizer of 
\[
\min_{z} \|H_{VV} + z z^{\top}\|_\mathrm{F}^2
\]
is given by the eigenvector associated with the smallest eigenvalue of \(H_{VV}\).  
By definition, this minimizer is exactly \(\ckopt\), and hence \(\ckopt\) is a local minimum of \(f(x)\) over \(C_V\).
thus $f(\ckopt)<f(c^{(l)})$ and using \cref{lemma:gap} we have:
\begin{equation}
    f\left(\ckopt\right)-f(c)<\frac{8n_z^2 \tau^2}{ \mu_1} \text {. }
\end{equation}
Because $f(c)$ attains the global minimum of $f$, the quantity 
\[
\min_{y \in \partial B_{\pm}} f(y) - f(c)
\] 
is a positive constant. If the compression threshold $\tau$ satisfies
\[
\tau < \frac{\sqrt{\mu_1 \bigl( \min_{y \in \partial B_{\pm}} f(y) - f(c) \bigr)}}{2 \sqrt{2} n_z},
\] 
then 
\begin{equation}
    f(\ckopt) - f(c) < \min_{y \in \partial B_{\pm}} f(y) - f(c),
\end{equation}
and $f(\ckopt)< \min_{y \in \partial B_{\pm}} f(y)$,  which implies that $\ckopt$ must lie inside $D^{\pm}$. 

To see this, suppose by contradiction that $\ckopt \notin D^{\pm}$. Then $\ckopt$ lies in $\mathbb{R}^n \setminus B_{\pm}$, where the minimum of $f$ can only occur either on the boundary $\partial B_{\pm}$ or at a stationary point of $f$. The stationary points outside $B_{\pm}$ are $\pm \sqrt{-\lambda'} x$ (with $\lambda' > \lambda$ and $x$ the corresponding eigenvector) and the zero vector. However, if
\begin{equation}
\begin{aligned}
\tau &< \min\Biggl\{
\frac{\sqrt{\mu_1 (f(\sqrt{-\lambda'} x) - f(c))}}{2\sqrt{2} n_z}, \,
\frac{\sqrt{\mu_1 (f(\mathbf{0}) - f(c))}}{2\sqrt{2} n_z} 
\Biggr\} \\
&= \min\Biggl\{
\frac{\sqrt{\mu_1 (\lambda^2 - \lambda'^2)}}{2\sqrt{2} n_z}, \,
\frac{\sqrt{\mu_1 \lambda^2}}{2\sqrt{2} n_z} 
\Biggr\},
\end{aligned}
\end{equation}
then we have
\[
f(\ckopt) 
< f(c) + \frac{8 n_z^2 \tau^2}{\mu_1}
< \min\!\left\{ f(\pm \sqrt{-\lambda'} x),\, f(\mathbf{0}) \right\}.
\]
Hence, $f(\ckopt)$ is smaller than the function values at all points in $\mathbb{R}^n \setminus B_{\pm}$, which contradicts the assumption that $\ckopt$ lies outside $D^{\pm}$. Therefore, $\ckopt \in D^{\pm}$.
\end{proof}

Since the iterates lie in a strongly convex region, we can now proceed to prove the linear convergence of the iteration.

\subsection{\cref{theorem:linear-convergence}'s proof}
\begin{proof}

We use $\Delta\bar{c}$ to represent the truncated vector obtained by taking the first $k$ components from  $\ckopt-c^{(l)}$,  and $\nabla \bar{f}(c^{(l)})$ to represent the truncated vector obtained by taking the first $k$ components of $\nabla {f}(c^{(l)}).$ Since the nonzero elements of $\ckopt$ and $c^{(l)}$ are all in first $k$ positions, we get 
$\nabla f(c^{(l)})^{\top}(\ckopt-c^{(l)})=\nabla \bar{f}(c^{(l)})^{\top}\Delta\bar{c}$ and $\Vert \ckopt-c^{(l)}\Vert_1=\Vert \Delta\bar{c}\Vert_1.$
Applying the strong convexity of $f$ in $ D^{\pm}$ again:
\begin{equation}
    \begin{aligned}
    f(\ckopt)&\geq f(c^{(l)})-(-\nabla f(c^{(l)})^{\top}(\ckopt-c^{(l)})-\frac{\mu_1}{2}\Vert \ckopt-c^{(l)}\Vert_1^2)\\
    &=f(c^{(l)})-(-\nabla \bar{f}(c^{(l)})^{\top}\Delta\bar{c}-\frac{\mu_1}{2}\Vert \Delta\bar{c}\Vert_1^2)\\
    &\geq f(c^{(l)})-(\Vert \nabla \bar{f}(c^{(l)})\Vert_{\infty}\Vert \Delta\bar{c}\Vert_1-\frac{\mu_1}{2}\Vert\Delta\bar{c}\Vert_1^2)\\
    &\geq f(c^{(l)})-\frac{1}{2\mu_1}\Vert \nabla \bar{f}(c^{(l)})\Vert_{\infty}^2\\
    &= f(c^{(l)})-\frac{1}{2\mu_1}\max_{1\leq i\leq k}\vert\nabla_if(c^{(l)})\vert^2\\
    &=f(c^{(l)})-\frac{1}{2\mu_1}(\nabla_{j_l}f(c^{(l)}))^2.
\end{aligned}
\end{equation}
The fourth inequality is obtained similarly by regarding the expression within the parentheses as a quadratic function of $\left\|\Delta\bar{c}\right\|_1$.Substituting $f(c^{(l+1)})\leq f(c^{(l)})-\frac{1}{2L}(\nabla_{j_l} f(c^{(l)}))^2$ obtained in \cref{lemma:gap} into the inequality, we get:
\begin{equation}
    f(c^{(l+1)})-f(\ckopt)\leq(1-\frac{\mu_1}{L})(f(c^{(l)})-f(\ckopt)).
\end{equation}

We have got the linear convergence of  function value. Since $D^{\pm}$ can be divided into two symmetric regions $D^{+}$ and $D^{-}$, the limiting point of the iteration must lie in one of them. Without loss of generality, we assume $c^{(l)} \in D^{+}$, and hence we may choose $\ckopt \in D^{+}$ as the convergence target. By the strong convexity of $f(x)$ in $D^{+}$, we obtain
\begin{equation}
\label{eq:c^l to ckopt}
    \begin{aligned}
    \Vert c^{(l)}-\ckopt\Vert_2^2
    &\leq \frac{2}{\mu_2}\Bigl(f(c^{(l)})-f(\ckopt)-\nabla f(\ckopt)^{\!\top}(c^{(l)}-\ckopt)\Bigr) \\
    &= \frac{2}{\mu_2}\bigl(f(c^{(l)})-f(\ckopt)\bigr) \\
    &\leq \frac{2}{\mu_2}\left(1-\frac{\mu_1}{L}\right)^l\bigl(f(c^{(0)})-f(\ckopt)\bigr)\;\;\longrightarrow\;0 .
    \end{aligned}
\end{equation}
The second equality holds because $\ckopt$ is a local minimizer of $f$ restricted to its support set $V$ (i.e., the first k indices). In particular, on the coordinates in $V$, the gradient vanishes, i.e., 
\[
\nabla_i f(\ckopt) = 0 \quad \text{for all } i \in V.
\]
Moreover, all nonzero entries of $(c^{(l)} - \ckopt)$ lie within this support set $V$. Therefore, the inner product 
\[
\nabla f(\ckopt)^T (c^{(l)} - \ckopt) = 0,
\]
and it follows that the iteration points converge to $\ckopt$.

On the other hand, if $c^{(l)} \in D^{-}$, then by symmetry we may take $-\ckopt \in D^{-}$ as the reference point. The same strong convexity argument applied in $D^{-}$ leads to an identical bound, and hence $\|c^{(l)} \mp \ckopt\|_2 \to 0$ depending on which region the iterates belong to.\qed
    
\end{proof}

Having presented the linear convergence result for the iterative sequence toward its stationary point \(\ckopt\), we next quantify the error in the estimated eigenvalue obtained from \(\ckopt\) relative to the true eigenvalue.

\section{\cref{lemma:exp-decay}'s proof}
\begin{proof}

We only prove the case for $s=1$, and the case for $s=2$ follows similarly. 
\begin{equation}
    \begin{aligned}
        \Vert x\Vert_1=\sum_{i=1}^n\vert (x)_i\vert&\leq N\Vert x\Vert_{\infty}+\sum_{j=N+1}^n\Vert x\Vert_{\infty}e^{-P(j-N)}\\
        &=N\Vert x\Vert_{\infty}+\left(\frac{e^{-P}-e^{-P(n-N+1)}}{1-e^{-P}}\right)\Vert x\Vert_{\infty}\\
        &\leq \left(N+\frac{e^{-P}}{1-e^{-P}}\right)\Vert x\Vert_\infty.
    \end{aligned}
\end{equation}
We let $Q_1= N+\frac{e^{-P}}{1-e^{-P}}$.
For case $s=2$, we can also get a constant $Q_2$ similarly. We let $Q=\max\{Q_1,Q_2\}$.
\end{proof}

\subsection{\cref{lemma:O(tau)}'s proof}

\begin{proof}
As shown in the proof of \cref{lemma:gap}, the convergence analysis implies that for sufficiently large $l$,
\[
\|\nabla f(c^{(l)})\|_{\infty}
= 4\left\|H c^{(l)} + \big(c^{(l)\top} c^{(l)}\big) c^{(l)}\right\|_{\infty}
< 4 n_z \tau.
\]
So, 
\[\left\|H c^{(l)}+\left(c^{(l) T} c^{(l)}\right) c^{(l)}\right\|_{\infty}<n_z \tau.\]
Since for every iterate \(c^{(l)}\), we have
\[
\big\|H c^{(l)} + (c^{(l)\top} c^{(l)}) c^{(l)}\big\|_{\infty} < n_z \tau,
\]
and \(c^{(l)} \to \ckopt\) as \(l\to\infty\) (as in \cref{eq:c^l to ckopt}), the same bound holds at the limit. Indeed, the mapping
\[
\Phi(x) := Hx + (x^\top x)\,x
\]
is continuous on \(\mathbb{R}^n\). Therefore
\[
\Phi(\ckopt) = \lim_{l\to\infty} \Phi(c^{(l)}),
\]
and taking the \(\ell_\infty\)-norm and passing to the limit yields
\[\left\|H \ckopt+\left(\ckoptT\ckopt\right)\ckopt\right\|_{\infty}\leq n_z \tau .\]
From above we  can get:
\begin{equation}
\begin{aligned}
\left|c^{\top}
\bigl(H\ckopt+(\ckoptT\ckopt)\ckopt\bigr)\right|
&\leq
\|c\|_1
\bigl\|H\ckopt+(\ckoptT\ckopt)\ckopt\bigr\|_{\infty} \leq n_z\|c\|_1\tau .
\end{aligned}
\end{equation}

Because $c$ is the eigenvector of $H$  corresponding to $\lambda$, $\ckopt$ is the eigenvector of $H_{11}$ and  $\Vert\ckopt\Vert_2=-\lambdakopt$, we have:
\begin{equation}
\begin{aligned}
\left|\lambda c^{\top}\ckopt-\lambdakopt c^{\top}\ckopt\right|
&=
\left|c^{\top}H\ckopt+(\ckoptT\ckopt)c^{\top}\ckopt\right| \leq n_z\|c\|_1\tau.
\end{aligned}
\end{equation}
By \cref{lemma:within-area} $\ckopt \in D^{ \pm}$, and let us denote
\begin{equation}
    d=\frac{1}{30} \frac{\min \left(-2 \lambda,-\lambda+\lambda_2\right)}{\sqrt{-\lambda}} .
\end{equation}
We can get $d<\frac{1}{2} \sqrt{-\lambda}=\frac{1}{2}\|c\|_2$  and 
\begin{equation}
c^{\top} \ckopt \geq c^{\top}(c-d\frac{c}{\Vert c\Vert_2}) =\|c\|_2^2-d\|c\|_2>\frac{1}{2}\|c\|_2^2=-\frac{\lambda}{2} . 
\end{equation}
So,
\begin{equation}
    \vert \Delta\lambda\vert<-\frac{2n_z\Vert c\Vert_1\tau}{\lambda}<\frac{2n_zQ\tau}{\sqrt{-\lambda}}.
\end{equation}

\end{proof}

Based on the $O(\tau)$ estimate of the eigenvalue error, we further analyze the bound of truncated components of the eigenvector. By controlling their contribution, we are able to derive a sharper $O(\tau^2)$ bound on the eigenvalue approximation error.

\subsection{\cref{lemma:norm-of-c2}'s proof}

\begin{proof}
Because $c$ is the eigenvector, we have
\begin{equation}
\label{eigenvector c}
\begin{aligned}
& H_{11} {c}_{1}+H_{12} {c}_{2}=\lambda {c}_{1}, \\
& H_{21} {c}_{1}+H_{22} {c}_{2}=\lambda {c}_{2}.
\end{aligned}
\end{equation}
We denote $c-\ckopt=\Delta c=\left(\begin{array}{c}
\Delta{c}_1 \\
c_2
\end{array}\right)$.
Using the compression conditions, we have:
\begin{equation}
\label{compression condition}
\begin{aligned}
& H_{11} \ckopt_{1}=\lambdakopt \ckopt_{1}, \\
& \vert(H_{21})_{ij} (\ckopt_{1})_j\vert<\tau;\quad \forall i>k ,j\leq k.
\end{aligned}
\end{equation} 
Combining \cref{eigenvector c} and \cref{compression condition} we have:
\begin{equation}
\begin{aligned}
H_{11} \Delta c_{1}+H_{12} c_{2}&=\lambda c_{1}-\lambdakopt \ckopt_{1}\\
&=\lambda c_1-\lambda\ckopt_1+\lambda\ckopt_1-\lambdakopt\ckopt_1\\
&=\lambda\Delta c_1+\Delta\lambda\ckopt_1.
\end{aligned}
\end{equation}
We also have 
\begin{equation}
    H_{21} \Delta c_{1}+H_{21} \ckopt_1+H_{22} c_{2}= H_{21} {c}_{1}+H_{22} {c}_{2}=\lambda {c}_{2}.
\end{equation}
These equation can be rearranged as follows: 
\begin{equation}
\label{equation of delta c}
\begin{aligned}
& H_{11} \Delta c_{1}+H_{12} c_{2}-\Delta\lambda\ckopt=\lambda \Delta c_{1}, \\
& H_{21} \Delta c_{1}+H_{22} c_{2}-\lambda c_2=-H_{21}\ckopt_1.
\end{aligned}
\end{equation}
Suppose that $\Delta c + x$ is an eigenvector of $H$ associated with the eigenvalue $\lambda$. 
Then $x$ satisfies
\begin{equation}
\label{equation of x}
    H \Delta c + H x = \lambda \Delta c + \lambda x .
\end{equation}

Combining \cref{equation of delta c} and \cref{equation of x} we have
\begin{equation}\label{eq:equation of x}
\begin{aligned}
(\lambda I-H)x=H\Delta c-\lambda\Delta c 
&=\left(\begin{array}{c}
H_{11} \Delta c_{1}+H_{12} c_{2}-\lambda\Delta c_1 \\
H_{21} \Delta c_{1}+H_{22} c_{2}-\lambda c_2
\end{array}\right)\\
&=\left(\begin{array}{c}
\Delta \lambda \ckopt_1 \\
-H_{21}\ckopt_1
\end{array}\right).
\end{aligned}
\end{equation}
The least squares solution of this equation is 
\begin{equation}
 x=(\lambda I-H)^\dagger \left(\begin{array}{c}
\Delta \lambda \ckopt_1 \\

-H_{21}\ckopt_1
\end{array}\right),
\end{equation}
where $(\lambda I-H)^\dagger$ is the pseudo-inverse of $\lambda I-H.$
Now we analyze the infinity norm of $x.$ By spanning $H_{21}\ckopt_1$ we can get
\begin{equation}
\begin{aligned}
\left\|H_{21} \ckopt_1\right\|_2^2 & =\sum_{i=k}^n\left(\sum_{j=1}^k H_{i j} c_j\right)^2  =\sum_{i=k}^n\left(\sum_{j_1, j_2=1}^k H_{i j_1} c_{j_1} H_{i j_2} c_{j_2}\right).
\end{aligned}
\end{equation}
We already have $\vert H_{i j_1} c_{j_1}\vert$ and $\vert H_{i j_2} c_{j_2}\vert<\tau$, and we need to count there are how many nonzero elements in $H_{ij_1}H_{ij_2}$.  First, choose $i$ and $j_1$, which gives most $kn_z$ possible choices. Then, $j_2$  has most $n_z$ options because we have fixed row index $i$, resulting in a total of $kn_z^2$ choices. So $\left\|H_{21} \tilde{c}_1\right\|_2^2\leq kn_z^2\tau^2.$
And we already know $\Vert \ckopt\Vert_2^2=-\lambdakopt$, and from \cref{lemma:O(tau)} we know that $\vert \Delta\lambda\vert<-\frac{2n_z\Vert c\Vert_1\tau}{\lambda}.$ Therefore,
\begin{equation}
\left\Vert
\binom{\Delta \lambda \tilde{c}_1}
      {-H_{21}\ckopt_1}
\right\Vert_2
<
m\sqrt{k}\,n_z\tau,
\qquad
m:=
\sqrt{\frac{-4\lambdakopt\|c\|_1^2}{k\lambda^2}+1}.
\end{equation}
Consequently,
\begin{equation}
\begin{aligned}
\|x\|_\infty
&\le \|x\|_2 \\
&\le
\left\|(\lambda I-H)^\dagger\right\|_2
\left\|
\binom{\Delta \lambda \tilde{c}_1}
      {-H_{21}\tilde{c}_1}
\right\|_2 \\
&<
\left\|(\lambda I-H)^\dagger\right\|_2
\,m n_z\sqrt{k}\,\tau \\
&=
\frac{m n_z\sqrt{k}\,\tau}
     {\operatorname{gap}(H)}.
\end{aligned}
\end{equation}
We have already assumed that  $\Delta c + x$ is an eigenvector of $H$
associated with the eigenvalue $\lambda$. 
Therefore, there exists a scalar $\alpha \in \mathbb{R}$ such that
\[
\Delta c + x = \alpha c .
\]
We proceed by contradiction. 
Assume that $\|c_2\|_\infty \ge 2\|x\|_\infty$.
Then there exists an index $j$ in the support of $c_2$ such that
\[
|(c_2)_j| = \|c_2\|_\infty .
\]
For this index, we have
\[
|(c_2)_j + x_j|
\ge |(c_2)_j| - |x_j|
\ge \frac{1}{2}|(c_2)_j|.
\]
On the other hand, since $\Delta c + x = \alpha c$, we also have
\[
|(c_2)_j + x_j| = |\alpha|\, |(c_2)_j|,
\]
which implies $|\alpha| \ge \frac{1}{2}$. Moreover, since $-\Delta c$ is also a solution of \eqref{eq:equation of x},
we obtain
\[
\|x\|_2 \le \|\Delta c\|_2 < d .
\]
Consequently,
\[
\|\Delta c + x\|_2 \le \|\Delta c\|_2 + \|x\|_2 < 2d .
\]
However, noting that $\|c\|_2 = \sqrt{-\lambda}$, the relation
$\Delta c + x = \alpha c$ together with $|\alpha| \ge \frac{1}{2}$
yields
\[
\|\Delta c + x\|_2 = |\alpha|\,\|c\|_2
\ge \frac{1}{2}\sqrt{-\lambda},
\]
which contradicts the assumption $2d < \frac{1}{2}\sqrt{-\lambda}$.
Therefore, the assumption $\|c_2\|_\infty \ge 2\|x\|_\infty$ is false, and we conclude that
\[
\|c_2\|_\infty < 2\|x\|_\infty
= \frac{2 m n_z \sqrt{k}\,\tau}{\operatorname{gap}(H)} .
\]

\end{proof}

\subsection{\cref{theorem: O(tau^2)}'s proof}
    
\begin{proof}
Recall that
\begin{equation}
\lambda c_1=H_{11}c_1+H_{12}c_2,\quad \lambdakopt\ckopt_1=H_{11}\ckopt_1.
\end{equation}
Subtracting the two equations, we have:
\begin{equation}
    \lambda c_1-\lambdakopt\ckopt_1-H_{11}(c_1-\ckopt_1)=H_{12}c_2 .
\end{equation}
Left-multiplying by $c^{\text{k-opt}^{\top}}$, we get:
\begin{equation*}
\lambda \ckoptT c_1-\lambdakopt\Vert \ckopt_1\Vert_2^2-\lambdakopt\ckoptT_1 c_1+\lambdakopt\Vert \ckopt_1\Vert_2^2=\ckoptT_1H_{12}c_2.
\end{equation*}
Rearranging this equation, we obtain
\begin{equation}
\label{equation of delta lambda}
    \Delta\lambda \ckoptT_1 c_1= (\lambda-\lambdakopt)\ckoptT_1 c_1=\ckoptT_1H_{12}c_2.
\end{equation}
Now we should analyze the upper bound of $\ckoptT_1H_{12}c_2$. Since 
\begin{equation}
\ckoptT_1H_{12}c_2=\sum_{i=1}^{k}\sum_{j=k}^n H_{ij}(\ckopt_1)_i(c_2)_j,
\end{equation}
and \[\vert H_{ij}(\ckopt_1)_i\vert<\tau,\quad\sum\limits_{i=1}^k\vert H_{ij}(\ckopt_1)_i\vert<n_z\tau, \] we have
\begin{equation}
    \vert\ckoptT_1H_{12}c_2\vert<\sum_{j=k}^{n}\vert n_z\tau\vert\vert (c_2)_i\vert.
\end{equation}
Since \(c_2\) is a subvector of \(c\), it inherits the exponential decay property of \(c\) and has the same exponential decay coefficient. Using the method similar to \cref{lemma:exp-decay}, we can derive
\begin{equation}
    \sum_{j=k}^{n}\vert (c_2)_i\vert<Q\Vert c_2\Vert_{\infty}=\frac{2mn_z\sqrt{k}Q\tau}{\operatorname{gap}(H)},
\end{equation}
and 
\begin{equation}
     \vert \ckoptT_1H_{12}c_2\vert<\frac{2mn_z^2\sqrt{k}Q\tau^2}{\operatorname{gap}(H)}.
\end{equation}
Moreover, from the estimate $d < \tfrac{1}{2}\|c\|_2$ established in the proof of \cref{lemma:O(tau)}, it follows that
\[
\ckoptT c \ge \|c\|_2^2 - d\|c\|_2 > \tfrac{1}{2}\|c\|_2^2 = -\tfrac{\lambda}{2}.
\]
From \cref{equation of delta lambda}, we obtain:
\begin{equation}
    \Delta\lambda< \frac{4mn_z^2\sqrt{k}Q\tau^2}{-\lambda\operatorname{gap}(H)}.
\end{equation}
    
\end{proof}

\section{Numerical Results}

In this section, we present numerical experiments to support our theoretical findings and demonstrate the practical performance of the compressed CDFCI algorithm. The experiments are organized into three parts. First, we verify the linear convergence using synthetic sparse matrices with exponentially decaying eigenvectors. We demonstrate that the convergence rate of the compressed algorithm matches that of the original uncompressed method. Second, we validate the error bound derived in the previous section. The results confirm that the numerical error remains consistently below the theoretical bound. Third, we apply the algorithm to realistic quantum chemistry systems. Exploiting the \( O(\tau^2) \) scaling of the error, we use extrapolation techniques on systems such as \(\text{H}_2\text{O}\) and \(\text{C}_2\) (with the cc-pVDZ basis set) to obtain improved estimates of the ground-state energy.

\subsection{Verification of Linear Convergence}

To construct test matrices with controlled spectral and sparsity properties, we first generate a set of orthonormal eigenvectors, ensuring that the first eigenvector (associated with the smallest eigenvalue) exhibits exponential decay. Random eigenvalues are then assigned to these eigenvectors, and the matrix is constructed as \(A = V\Lambda V^{\top}\). To enforce sparsity, we truncate entries with small magnitudes in the resulting matrix. After sparsification, we verify whether the first eigenvector still maintains exponential decay. If not, we re-impose this structure by element-wise multiplication with an exponential decay factor and repeat the construction procedure until the desired properties are satisfied.

For all experiments, the initial iterate is chosen as the target eigenvector perturbed by a random noise vector. This initialization ensures that the iterates start in a neighborhood of the desired solution while avoiding the trivial case of starting from the exact eigenvector.

\Cref{fig:convergence_tau} displays the convergence behavior of CDFCI under three different compression thresholds: \(\tau = 0\) (uncompressed), \(10^{-5}\), and \(10^{-7}\). The horizontal axis represents the iteration count, and the vertical axis plots the logarithm of the objective gap, \(\log(f(c^{(l)}) - f(\ckopt))\). In each subplot, the red line represents a linear fit, highlighting the linear convergence rate.
\begin{figure}[htbp]
    \centering
    \begin{subfigure}[b]{0.32\textwidth}
        \includegraphics[width=\textwidth]{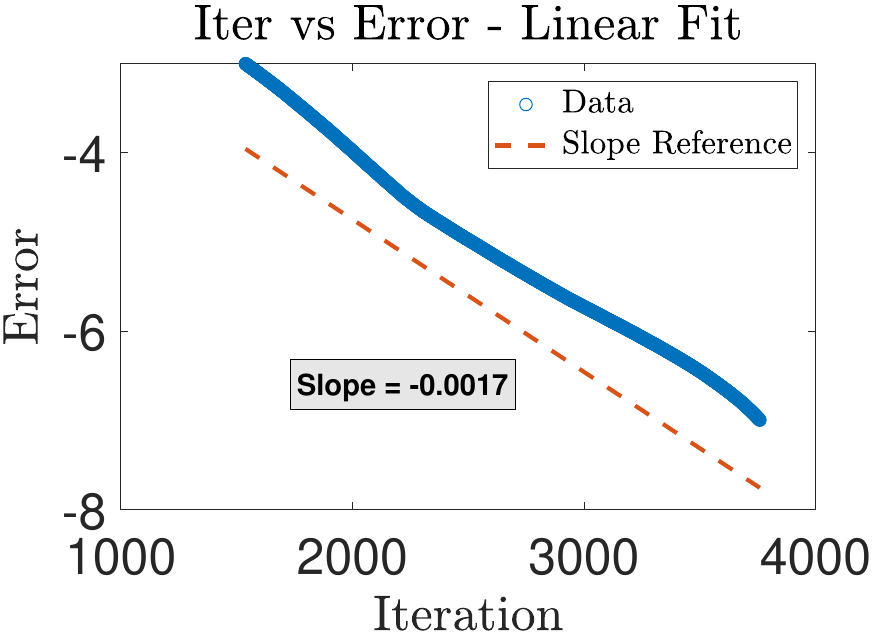}
        \caption{$\tau = 0$}
    \end{subfigure}
    \hfill
    \begin{subfigure}[b]{0.32\textwidth}
        \includegraphics[width=\textwidth]{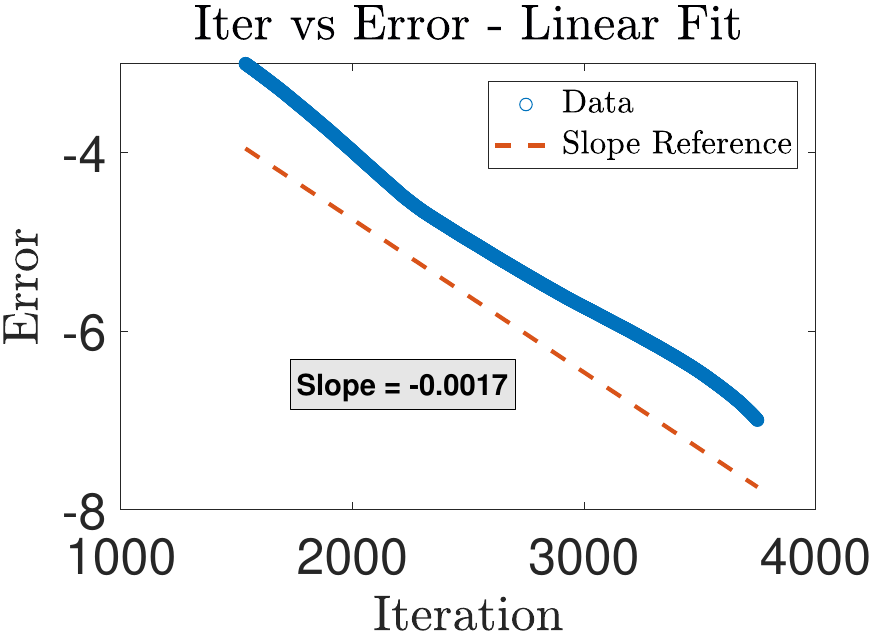}
        \caption{$\tau = 10^{-5}$}
    \end{subfigure}
    \hfill
    \begin{subfigure}[b]{0.32\textwidth}
        \includegraphics[width=\textwidth]{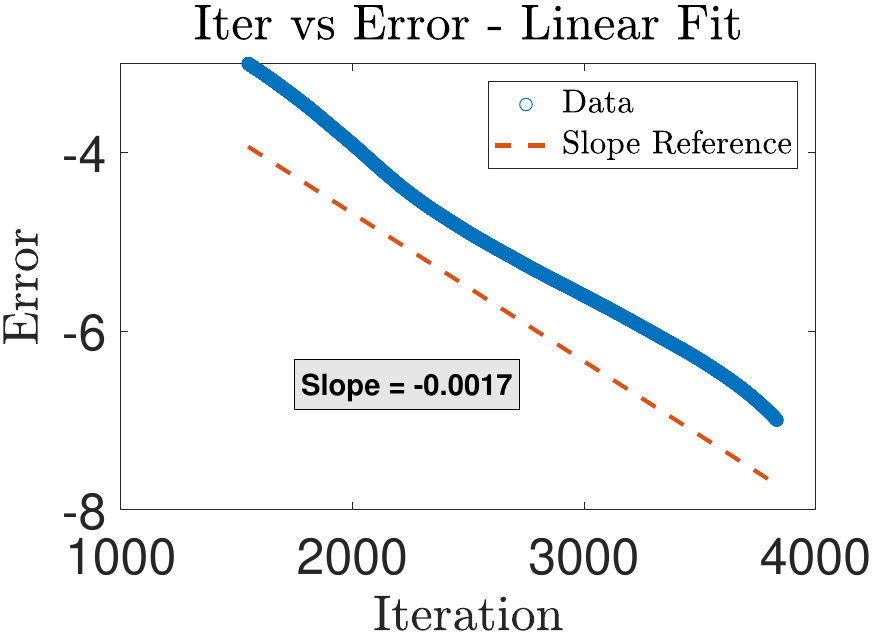}
        \caption{$\tau = 10^{-7}$}
    \end{subfigure}
    \caption{Convergence behavior under different compression thresholds. The horizontal axis represents the iteration number, and the vertical axis plots the logarithmic objective gap $\log(f(c^{(l)}) - f(\ckopt))$. The red lines indicate the linear fit.}
    \label{fig:convergence_tau}
\end{figure}

As shown in \cref{fig:convergence_tau}, the algorithm maintains linear convergence even when compression is applied. Moreover, the observed convergence rates for the compressed cases are nearly identical to the uncompressed version, indicating that the compression strategy does not compromise the convergence speed.

\subsection{Verification of the Error Bound}

In this subsection, we investigate the tightness of the theoretical error bound derived in Section 3. We first consider the same synthetic matrix class used in the previous subsection, referring to it as \textbf{Type 1}. To further evaluate the error bound under different sparsity patterns, we introduce a second class of matrices, referred to as \textbf{Type 2}.

A Type 2 matrix consists of two main blocks: a diagonal block in the top-left and a tridiagonal block in the bottom-right, coupled by a small number of off-diagonal entries. This structure preserves the exponential decay of the first eigenvector while significantly increasing the overall sparsity of the matrix. \Cref{fig:type2_spy} illustrates the sparsity pattern of a representative Type 2 matrix, where the block structure and coupling entries are clearly visible.

\begin{figure}[htbp]
    \centering
    \includegraphics[width=\stdfigwidth]{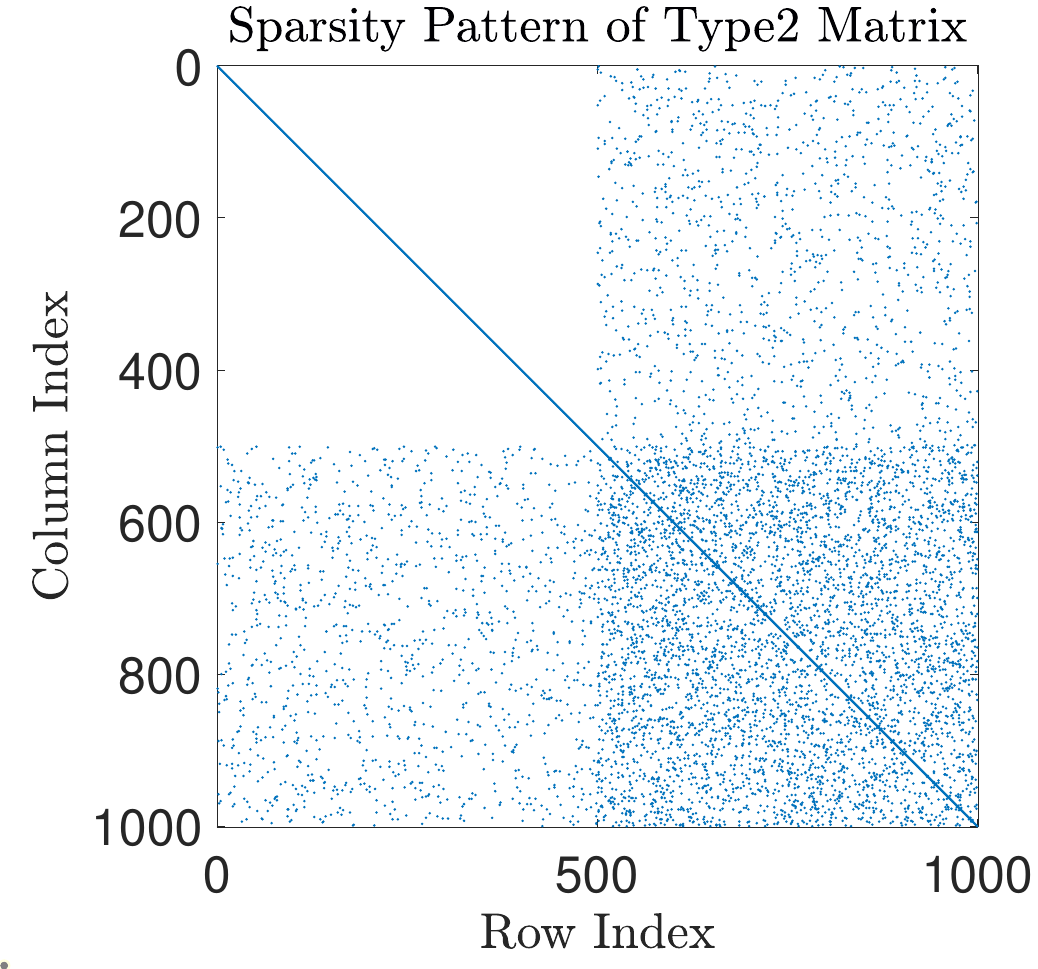}
    \caption{Sparsity pattern of a representative Type 2 matrix.}
    \label{fig:type2_spy}
\end{figure}

\Cref{fig:err_bound_compare} presents a comparison between the actual Rayleigh quotient error and the theoretical bound for both matrix types. In both cases, the actual error remains consistently below the predicted bound and exhibits a similar decay pattern as \(\tau\) decreases. Notably, the Type 2 matrix yields a significantly tighter bound: the typical ratio between the bound and the true error ranges from \(10^1\) to \(10^2\), whereas for the Type 1 matrix, this ratio ranges from \(10^2\) to \(10^3\).

\begin{figure}[htbp]
    \centering
    \begin{subfigure}[b]{\stdfigwidth}
        \centering
        \includegraphics[width=\stdfigwidth]{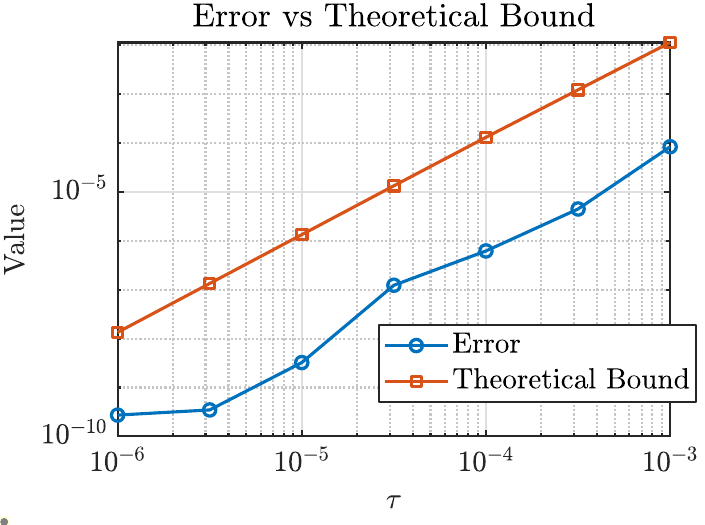}
        \caption{Type 1 matrix}
        \label{fig:err_bound_old}
    \end{subfigure}
    \hfill
    \begin{subfigure}[b]{\stdfigwidth}
        \centering
        \includegraphics[width=\stdfigwidth]{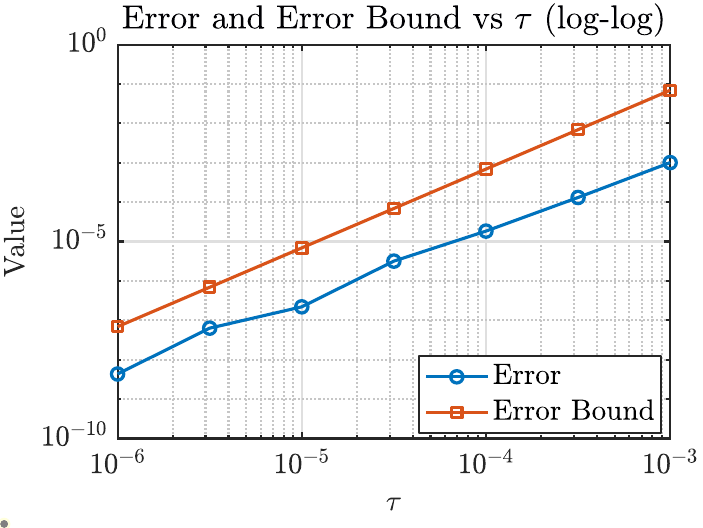}
        \caption{Type 2 matrix}
        \label{fig:err_bound_new}
    \end{subfigure}
    \caption{Comparison of the actual error and the theoretical bound under varying \(\tau\) for the two types of synthetic matrices.}
    \label{fig:err_bound_compare}
\end{figure}

Next, we analyze the compression error and the theoretical error bound for the \(\text{H}_2\text{O}\) molecule using the cc-pVDZ basis set. The behavior observed in \cref{fig:h2o_err} aligns with the results from our synthetic test matrices: the actual error consistently satisfies the theoretical bound and exhibits a similar decay pattern as \(\tau\) decreases. However, the ratio between the bound and the true error is larger, typically ranging from \(10^3\) to \(10^4\). This looseness in the bound may be attributed to the extremely high sparsity characteristic of FCI Hamiltonians in quantum chemistry problems.

\begin{figure}[htbp]
    \centering
    \includegraphics[width=\stdfigwidth]{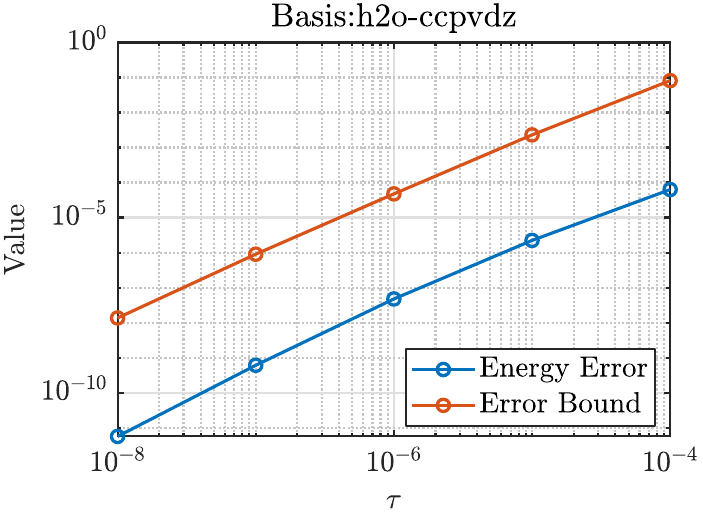}
    \caption{Compression error and theoretical error bound for \(\text{H}_2\text{O}\) (cc-pVDZ).}
    \label{fig:h2o_err}
\end{figure}

\subsection{Extrapolation of Ground-State Energy via Compression Error}

In this subsection, we demonstrate how to leverage the \( O(\tau^2) \) quadratic dependence of the energy error on the compression threshold to obtain improved estimates of the ground-state energy. Even when the exact Full Configuration Interaction (FCI) solution is computationally inaccessible, this scaling law enables us to extrapolate the ground-state energy by fitting results at different values of $\tau$.

The extrapolation scheme is based on the theoretical finding that the compression error scales quadratically with \(\tau\). Specifically, we adopt the following model (ansatz):
\begin{equation}
    E(\tau) = E_0 + a \tau^2,
\end{equation}
where \( E(\tau) \) denotes the computed ground-state energy at threshold \(\tau\), \( E_0 \) represents the extrapolated true energy (corresponding to the limit \(\tau \to 0\)), and \( a \) is a fitting coefficient.

To reduce the impact of ill-conditioning, we rewrite the model in logarithmic form:
\begin{equation}
    \log(E(\tau) - E_0) = \log(a) + 2 \log(\tau).
\end{equation}
Given this relation, we employ an iterative procedure: we first select a preliminary estimate for \( E_0 \), and then solve a linear least-squares problem to find the best-fit parameter \( a \). Subsequently, we refine \( E_0 \) via a bisection search to minimize the fitting residual in the logarithmic domain. This process is repeated until convergence, yielding an accurate extrapolated estimate of the true ground-state energy.

We apply this extrapolation strategy to quantum chemistry systems using standard basis sets, specifically \(\text{H}_2\text{O}\) and \(\text{C}_2\) with cc-pVDZ. To evaluate the robustness of the method, we consider two distinct threshold regimes: a high-precision range (\(10^{-8}\text{--}10^{-7}\)) and a low-precision range (\(10^{-5}\text{--}10^{-4}\)). In each range, we uniformly sample 10 values of \(\tau\). This allows us to examine the extrapolation accuracy in both high- and low-precision regimes.

\begin{table}[htbp]
\centering
\caption{Comparison of extrapolation performance under different compression threshold regimes.}
\label{tab:extrapolation}
\resizebox{\textwidth}{!}{
    \begin{tabular}{lcc}
    \toprule
    \textbf{System} & \textbf{Min. Error (Computed)} & \textbf{Error (Extrapolated)}  \\
    \midrule
    \multicolumn{3}{l}{\textit{Small \(\tau\) regime (\(10^{-8}\) to \(10^{-7}\))}} \\
    \(\text{H}_2\text{O}\) (cc-pVDZ) & \(5.6 \times 10^{-12}\) & \(1.1 \times 10^{-12}\) \\
    \(\text{C}_2\) (cc-pVDZ)   & \(2.0 \times 10^{-10}\) & \(2.6 \times 10^{-11}\) \\
    \midrule
    \multicolumn{3}{l}{\textit{Large \(\tau\) regime (\(10^{-5}\) to \(10^{-4}\))}} \\
    \(\text{H}_2\text{O}\) (cc-pVDZ) & \(2.2 \times 10^{-6}\)  & \(7.0 \times 10^{-7}\)  \\
    \(\text{C}_2\) (cc-pVDZ)   & \(6.5 \times 10^{-5}\)  & \(3.4 \times 10^{-5}\)  \\
    \bottomrule
    \end{tabular}
}
\end{table}

The results demonstrate that extrapolation yields significant accuracy improvements when the original data is in the high-precision regime (small \(\tau\)). However, for data with lower precision (large \(\tau\)), the improvement is less pronounced. This is expected, as the error asymptotically approaches the quadratic scaling \( O(\tau^2) \) only as \(\tau \to 0\). Consequently, the quadratic model provides a much better fit in the small \(\tau\) regime. Developing extrapolation techniques that remain effective for larger \(\tau\) values represents an important direction for future research.

\section{Conclusion}

In this work, we established a theoretical framework for the CDFCI with compression algorithm. Under mild assumptions on the spectral gap and exponential decay of the ground-state eigenvector, we proved that the algorithm maintains linear convergence toward the compressed stationary point and derived an error bound of order $O(\tau^2)$ for the Rayleigh quotient. Numerical experiments on both synthetic sparse matrices and realistic quantum chemistry systems confirm the theoretical predictions: the algorithm exhibits consistent linear convergence and achieves energy errors in agreement with the $O(\tau^2)$ bound. Furthermore, by exploiting the theoretically characterized dependence of the error on $\tau^2$, we proposed an extrapolation scheme that yields significantly improved energy estimates, demonstrating the practical value of the derived error analysis.

\section*{Acknowledgments}

The authors acknowledge Jianfeng Lu and Zhe Wang for their contributions to the original development of the CDFCI method. This work was supported in part by the National Natural Science Foundation of China under Grant Nos.~12271109 and 12526211; the Shanghai Pilot Program for Basic Research--Fudan University under Grant No.~21TQ1400100 (22TQ017); the Scientific Research Innovation Capability Support Project for Young Faculty under Grant No.~SRICSPYF-ZY2025159; and the Xuemin Institute of Advanced Studies, Fudan University.
\bibliographystyle{siamplain}
\bibliography{CDFCI}

\end{document}